\documentclass{birkjour}
\usepackage{amsmath,amssymb,amsthm,verbatim}
\usepackage{color,graphics,srcltx}
\usepackage{xcolor}
\theoremstyle{plain}
\newtheorem{theorem}{Theorem}[section]
\newtheorem{lemma}[theorem]{Lemma}

\newtheorem{proposition}[theorem]{Proposition}

\numberwithin{equation}{section}
\allowdisplaybreaks[1]

\theoremstyle{definition}

\theoremstyle{remark}

\newcommand{\sbm}[1]{\left[\begin{smallmatrix} #1
                \end{smallmatrix}\right]}
\newcommand{\spm}[1]{\left(\begin{smallmatrix} #1
                \end{smallmatrix}\right)}
\makeatletter
\def\revddots{\mathinner{\mkern1mu\raise\p@
\vbox{\kern7\p@\hbox{.}}\mkern2mu
\raise4\p@\hbox{.}\mkern2mu\raise7\p@\hbox{.}\mkern1mu}}
\makeatother

\def\IF{{\mathbb F}}

\def\IC{{\mathbb C}}
\def\IR{{\mathbb R}}

\def\cB{{\mathcal B}}

\def\cR{{\mathcal R}}

\def\cH{{\mathcal H}}
\def\cT{{\mathcal T}}

\def\be{{\mathbf e}}

\def\b1{{\mathbf 1}}

\def\char{{\rm char}}
\begin{document}
\title[Preserver problems]{Preserver problems on Toeplitz matrices}
\author[R. Ahmed]{Rayhan Ahmed}
\email{rahmed03@wm.edu}
\author[V. Bolotnikov]{Vladimir Bolotnikov}
\email{vladi@math.wm.edu}
\author[W. Hoyle]{William Hoyle}
\email{whoyle@wm.edu}
\author[C.-K. Li]{Chi-Kwong Li}
\email{ckli@math.wm.edu}
\address{Department of Mathematics,  William \& Mary,
Williamsburg VA 23187-8795, USA}
\begin{abstract}
We study linear preserver problems on the linear space of $n\times n$ 
Toeplitz matrices over the real field or the complex field. 
In particular, characterizations are given 
for linear preservers of rank one matrices and linear preservers of
the determinant. We also present related results and questions on other structured 
matrices.
\end{abstract}

\dedicatory{In memory of Leiba Rodman}
\maketitle

\section{Introduction}

Linear preserver problems form a classical and active area of matrix theory and 
operator algebras. They concern the characterization of linear maps on spaces of 
matrices (or operators) that leave invariant certain properties, functions, sets 
or relations \cite{LP,LT}. 
The subject originated with Frobenius \cite{Fr} who determined 
the precise form of linear preservers of the determinant on $M_n$, the set of $n\times n$ complex matrices. 
It was shown that such a map has the form 
\begin{equation}
A \mapsto M A N\quad\mbox{or}\quad A \mapsto M A^\top N
\label{class}
\end{equation}
with $\det(MN)=1$. The result was extended to matrices of arbitrary fields by 
Dieudonn\'{e} in \cite{D}. Marcus and Moyls in \cite{MM1,MM2} determined the structure 
of linear rank preservers 
The preservers also have the form \eqref{class} for some nonsingular matrices $M$ and $N$.
It turns out that many preservers  can be shown to be rank one preservers so that 
the result of Marcus and Moyls applies. The preserving properties can then be used
to deduce additional information on the matrices $M$ and $N$.

\smallskip

In this paper, we study linear preserver problems in the linear space of $n\times n$
Toeplitz matrices over $\IF$,  where $\mathbb F$ is the complex field $\mathbb C$ or 
the real field $\mathbb R$.  Recall that 
a {\em Toeplitz matrix} (or diagonal-constant matrix) is a matrix in which each descending 
diagonal from left to right is constant. Formally, 
a matrix $A\in M_n$ is Toeplitz if its
$(i,j)$ entry 
$A_{i,j}$ depends only on the index difference, i.e., $A_{i,j} = 
a_{i-j}$ for $a_{1-n}, \dots, a_0, \dots, a_{n-1} \in \IF$.
Hence,
\[
A = \begin{bmatrix} 
a_0 & a_{-1} & a_{-2} & \dots & a_{-(n-1)} \\
a_1 & a_0 & a_{-1} & \dots & a_{-(n-2)} \\
a_2 & a_1 & a_0 & \dots & a_{-(n-3)} \\
\vdots & \vdots & \vdots & \ddots & \vdots \\
a_{n-1} & a_{n-2} & a_{n-3} & \dots & a_0
\end{bmatrix} . \]
Toeplitz matrices are useful in many topics including signal processing, numerical
computation and  operator theory; for example, see \cite{Toe1, Toe2, Toe3} and their references.
The study of preservers on Toeplitz matrices reveals new rigidity phenomena arising from the interplay between 
linearity, structure, and the preserved quantity. 

\smallskip

In this paper, we characterize  
linear maps on Toeplitz matrices preserving the determinant function
and the rank function; see Section 4. 
These are established after we obtain the basic 
result on linear preservers of rank one (Toeplitz) matrices in Section 3.
The proof depends on some basic properties of Toeplitz matrices and rank one 
Toeplitz matrices established in Section 2.  In Section 4, we present additional 
preserver results for Toeplitz matrices. 
We also discuss some related results, and further questions.

\section{Preliminary results}

Let ${\cT}_n$ be the linear space of $n\times n$ Toeplitz matrices on $\IC$ or $\IR$.
Let ${\bf e}_k$ denote the $k$-th column of the identity matrix 
$I_n=\begin{bmatrix}{\bf e}_1&\ldots&{\bf e}_n\end{bmatrix}$ and 
let $Z_n$ be the lower shift matrix defined as 
$$
Z_n=\begin{bmatrix}{\bf e}_2&\ldots&{\bf e}_n&0\end{bmatrix} = \begin{bmatrix} 
0 & & & \cr
1 & 0 & &\cr
   & \ddots & \ddots & \cr
   & & 1 & 0 \cr\end{bmatrix}.
$$ 
It is clear that $\dim {\cT}_n=2n-1$ and that
\begin{equation}
\cB = \left\{Z_n^{n-1}, \; Z_n^{n-2}, \; \ldots, \; Z_n, \; I_n, \;  Z_n^{\top},  \, \ldots, \,  
(Z^{\top}_n)^{n-2}, \, (Z^{\top}_n)^{n-1}\right\}
\label{basis}
\end{equation}
is a basis for ${\cT}_n$. We have the following characterization of rank-one Toeplitz matrices.

\begin{proposition}
A matrix $A\in {\cT}_n$ has rank one if and only if 
there are $\mu, \xi \in \IF$ with $\mu \ne 0$ such that 
$A$ equals
\begin{equation}
A_1=\mu {\bf e}_1{\bf e}_n^\top\quad\mbox{or}\quad
A_2= \mu
\sbm{\xi^{n-1}  \vspace{-1mm}\\ 
\vdots \\ \xi \\ 1}
\sbm{1&\xi &\ldots &\xi^{n-1}}.
\label{1}
\end{equation}
\label{P:1.1}
\end{proposition}
\begin{proof}
The ``if" part is obvious. For the ``only if" part, take a Toeplitz rank-one matrix in the form
\begin{equation}
A={\bf b}^\top {\bf c}, \; \; \mbox{where}\; \;  {\bf b}=\begin{bmatrix}b_1 & \ldots & b_n\end{bmatrix}\neq 0,
\quad {\bf c}=\begin{bmatrix}c_1 & \ldots & c_n\end{bmatrix}\neq 0.
\label{2}
\end{equation}
If $c_1=0$, let $c_k$ be the leftmost nonzero entry in ${\bf c}$. Then the $k-1$ leftmost columns in $A$ are zeros. Then the 
Toeplitz structure of $A$ guarantees that $b_ic_k=0$ for $i=2,\ldots,n$, 
and hence, $b_i=0$ for $i\ge 2$. Thus, ${\bf b}=b_1{\bf e}_1^\top$.  
Subsequently, only the top row in $A$ is non-zero, which is possible (again by the Toeplitz structure of $A$) only if $k=n$, i.e., if 
$c_n$ is the only non-zero entry in ${\bf c}$. Thus, ${\bf c}=c_n{\bf e}_n^\top$ which together with 
\eqref{2}, leads us to $A=A_1$ with $\mu=b_1c_n$.

\smallskip

If $c_1\neq 0$, we let $\xi:=\frac{c_2}{c_1}$.
Since $A$ is a Toeplitz matrix,  we conclude from equalities $a_{j,1}=a_{j+1,2}$ 
that $b_j=b_{j+1}\xi$ for $j=1,\ldots,n-1$. Hence, 
$$
{\bf b}=b_n\begin{bmatrix} \xi^{n-1} & \xi^{n-2}&\ldots \xi&1\end{bmatrix}.
$$
Since ${\bf b}\neq 0$, we have $b_n\neq 0$. Since the diagonal entries in $A$ 
are equal, we have
$$
\xi^{n-1}c_1=\xi^{n-2}c_2=\ldots=\xi c_{n-1}=c_n,
$$
which implies ${\bf c}=c_1\begin{bmatrix} 1& \xi &\ldots \xi^{n-2} & \xi^{n-1}\end{bmatrix}$. Substituting ${\bf b}$ and 
${\bf c}$ as above into \eqref{2} we conclude that $A=A_2$ with $\mu=b_nc_1$.
\end{proof}

It is easy to see that with respect to the basis $\cB$ defined in 
(\ref{basis}), the coordinate vectors of rank one Toeplitz matrices in $\IF^{2n-1}$ have the form 
\begin{equation}
\label{h0}
\mu [0, \dots, 0, 1]^\top \quad \hbox{ or } \quad \mu [1, \xi, \dots, \dots, \xi^{2n-2}]^\top
\end{equation}
for some $\mu, \xi \in \IF$ with $\mu \ne 0$. Note also that 
$$
[0,\dots, 0,1] = \lim_{\xi \rightarrow \infty} \frac{1}{\xi^{2n-2}}[1, \dots, \xi^{2n-2}].$$ 
So, the first vector in \eqref{h0} can be viewed as the limit of a vector of the second type.  
In the following discussion, we use the notation
\begin{equation}
\label{h} 
h(\infty) = [0, \dots, 0, 1]^\top \quad \hbox{ and} \quad
 h(\xi) =  [1, \xi, \dots, \xi^{2n-2}]^\top \ \hbox{ with } \xi \in \IF.
\end{equation}
Moreover, we let $\mathcal R$ be the set of all non-zero multiples of vectors \eqref{h}:
\begin{equation}
\mathcal R= 
\{ \mu h(\xi): \mu \in \IF \setminus \{0\}, \  \xi \in \IF \cup \{\infty  \}\}\label{calr}
\end{equation}
It is easy to verify that a nonzero vector 
$h=\begin{bmatrix} h_1&\ldots&h_{2n-1}\end{bmatrix}^\top$ lies in  
$\mathcal R$ if and only if
\begin{equation} \label{calr-2}
h_j^2=h_{j-1}h_{j+1} \quad  \hbox{ for } \quad   j=2,\ldots,2n-2.
\end{equation}
We have the following.

\begin{proposition} 
\label{R:1.3} 
Let $k\le 2n-1$. A set of $k$ nonzero vectors in $\cR$ is linearly dependent 
if and only if two of the $k$ vectors are multiples
of each other.
\end{proposition}

\begin{proof} The sufficiency is clear. For the contrapositive proof of necessity,  
assume that no two vectors are multiples of each other.
By normalization, we may assume all given $k$ vectors are of the form $h(\xi)$. We use
these vectors to construct the $(2n-1)\times k$ matrix
$$
R = \begin{bmatrix} h(\xi_1) & \ldots & h(\xi_{k})\end{bmatrix},\quad \xi_i\in\mathbb F\cup\{\infty\},
$$
in which no two columns are identical.

\smallskip

If $\xi_i\neq\infty$ for $i=1,\ldots,k$, then $R$ is a $(2n-1)\times k$ Vandermonde matrix.
The matrix $R_1$ obtained by removing the last $2n-k-1$ rows from $R$
is a square Vandermonde matrix with determinant $\prod_{1 \le i < j \le k} (\xi_i-\xi_j) \ne 0$.
Hence, $R_1$ has rank $k$, and so does $R$. Thus, the columns in $R$ are linearly independent.

\smallskip

If the column $h(\infty)$ appears in $R$, we may assume that it is the rightmost column.
Let $R_1$ be the matrix obtained by removing the rows with indices $k, \dots, 2n-2$.
Then the last column of  $R_1$ equals $[0, \dots, 0, 1]^\top\in \IF^k$ and the leading $(k-1)\times (k-1)$ submatrix $R_2$
is again a Vandermonde matrix with determinant $\prod_{1 \le i < j \le k-1} (\xi_i-\xi_j) \ne 0$.
Since the last column of $R_1$ is $[0, \dots, 0, 1]^\top$, we see that $\det(R_1) = \det(R_2)$. Hence, $R_1$ has rank $k$, and so does
$R$. Thus, the columns in $R$ are linearly independent.
\end{proof}

\section{Rank one preservers}

In this section, we determine the structure of linear maps $T:{\cT}_n \to {\cT}_n$ sending rank one 
matrices to rank one matrices. Recall that 
$$\cB = \left\{Z_n^{n-1}, \; Z_n^{n-2}, \; \ldots, \; Z_n, \; I_n, \;  Z_n^{\top},  \, \ldots, \,  
(Z^{\top}_n)^{n-2}, \, (Z^{\top}_n)^{n-1}\right\}$$
and 
$$\cR = \{[0, \dots, 0, \mu]^\top: \mu \in \IF \setminus \{0\}\}
\cup \{ \mu [1, \xi, \dots, \xi^{2n-2}]^\top: \mu \in \IF \setminus \{0\}, \  \xi \in \IF \}.
$$
Let $L=\left[T\right]_{\cB}\in M_{2n-1}$ be the matrix of a linear map $T: \cT_n \to \cT_n$ with respect to the basis $\cB$.
By Proposition \ref{P:1.1}, $T$ preserves rank one matrices if and only if $Lh\in\mathcal R$ for all $h\in\mathcal R$. 
With a slight abuse of notation, we will write $L(\cR) \subseteq \cR$ to denote the latter invariance property. Also, if 
$h \in \IF^{2n-1}$, we will not distinguish $L(h)$ and $Lh$. Furthermore, 
we will use notation $F_n$ and $D_n(r)$ for the anti-diagonal permutation matrix and the diagonal matrix given by
\begin{equation}
F_n=\begin{bmatrix}0&\ldots&0&1\\
\vdots&\revddots&\revddots&0\\
0&1&\revddots& \vdots\\ 1&0&\ldots&0\end{bmatrix},\quad 
D_n(r)=\begin{bmatrix} 1 &0&\ldots &0\\ 0&r&\ddots&\vdots\\ \vdots&\ddots&\ddots& 0\\
0&\ldots&0&r^{n-1}\end{bmatrix}.
\label{a}
\end{equation}
\begin{lemma} 
\label{lem0} 
The set of matrices $L \in M_{2n-1}$ that satisfy $L(\cR) \subseteq \cR$ 
is closed under multiplication and contains $F_{2n-1}$, $r I_{2n-1}$
and  $D_{2n-1}(r)$ for any nonzero $r \in \IF$.
\end{lemma}

\begin{proof} The first statement is clear.

If $L = F_{2n-1}$, then for any nonzero $\xi\in \IF$,
$L(h(\xi)) = \mu h(y)$ with $\mu=\xi^{2n-2}$ and $y=1/\xi$.
Moreover, $L(h(\infty)) = h(0)$ and $L(h(0)) = h(\infty)$.

\smallskip

Suppose $r \in \IF$ is nonzero. If $L = rI_{2n-1}$, then $L(\cR) = \cR$.
If $L  = D_{2n-1}(r)$, 
then for any $\xi \in \IF$, $L(h(\xi)) = h(y)$ with $y = r\xi$.
In addition, $L(h(\infty)) = r^{2n-2} h(\infty)$.
In all cases, we see that  $L(\cR) \subseteq \cR$.
\end{proof}
Note that it is essential to assume that $r \in \IF$ 
is nonzero in the above lemma.
If $r = 0$, then $L(h(\infty)) = 0 \notin \cR$ for $L = rI_{2n-1}$ or $L = D_{2n-1}(r)$.

\smallskip

We will identify more matrices $L \in M_{2n-1}$ satisfying $L(\cR) \subseteq \cR$.
For $\alpha \in \IF$, let 
\begin{equation}
V_n(\alpha):=\begin{bmatrix}1&\alpha & \alpha^2 &\ldots& \alpha^{n-1}\\
0 & 1 & 2\alpha  &\ldots & \spm{n-1\\ 1}\alpha^{n-2}\\
0&0&1&\ldots& \spm{n-1\\ 2}\alpha^{n-3}\\
\vdots & \ddots & \ddots&\ddots &\vdots \\
0 &\ldots & 0 &\ldots&  1 \end{bmatrix},
\label{3}
\end{equation}
which is known as the confluent Vandermonde matrix associated 
with $\alpha$. The study of confluent Vandermonde matrices is
connected to many different areas; e.g., see \cite{bconf,LiLin} and references therein.
Denote by
\begin{equation} \label{J-alpha}
\mathcal J_\alpha=\alpha I_n+Z_n = 
\begin{bmatrix} 
\alpha & & & \cr
  1 & \alpha & & \cr
  & \ddots & \ddots & \cr
  & & 1 & \alpha\cr\end{bmatrix}
  \end{equation} 
the lower triangular Jordan block
with $\alpha$ on the main diagonal. Then we can verify that the 
$j$-th column of $V_n(\alpha)$ is equal to $\mathcal J_\alpha^{j-1}{\bf e}_1$, i.e., 
\begin{equation}
V_n(\alpha)=\begin{bmatrix}{\bf e}_1 & \mathcal J_\alpha{\bf e}_1&\ldots & \mathcal J_\alpha^{n-1}{\bf e}_1\end{bmatrix}.
\label{4}
\end{equation}

The next statement presents a characteristic property of $V_n(\alpha)$.

\begin{lemma}
\label{L:2.1}
Let $\alpha \in \IF$. The matrix \eqref{3} is the unique matrix that satisfies the identity
\begin{equation}
\begin{bmatrix}1&x&\ldots&x^{n-1}\end{bmatrix}V_n(\alpha)
=\begin{bmatrix}1&x+\alpha&\ldots&(x+\alpha)^{n-1}\end{bmatrix}\quad\mbox{for all} \;\; x.
\label{6}
\end{equation}
In particular, the matrix 
$L = V_{2n-1}(\alpha)^\top\in M_{2n-1}$ satisfies $L(\cR) \subseteq \cR$ .
\end{lemma}

\begin{proof}
By \eqref{4}, the equality \eqref{6} is equivalent to equalities
\begin{equation}
\begin{bmatrix}1&x&\ldots&x^{n-1}\end{bmatrix}\mathcal J_\alpha^j{\bf e}_1 =(x+\alpha)^j
\quad \mbox{for}\quad j=0,\ldots, n-1,
\label{6a}
\end{equation}
which are immediate by \eqref{3} and the Binomial Theorem. On the other hand, equalities \eqref{6a} uniquely determine
the entries in $\mathcal J_\alpha^j{\bf e}_1$ as the coefficients of $x^k$ in the expansion of $(x+\alpha)^j$. 

\smallskip

Clearly, if $L: = V_{2n-1}(\alpha)^\top$, then 
$L(h(\infty)) = h(\infty)$, and for any $\xi \in \IF$,
$Lh(\xi)=h(y)\in \cR$ with $y = \xi+\alpha$. So, $L(\cR) \subseteq \cR$.
\end{proof}
\noindent
Next, we introduce another type of matrices which  preserve $\cR$.

\begin{lemma} 
\label{L:2.2}
Let   
$\mathcal J_\alpha=\alpha I_n+Z_n$ and $\mathcal J_\beta=\beta I_n+Z_n$ with $\alpha, \beta \in \IF$. 
Then 
\begin{equation}
\begin{bmatrix}1&x&\ldots&x^{n-1}\end{bmatrix}\mathcal J_\beta^{k}\mathcal J_\alpha^{\ell}{\bf e}_1
=(x+\beta)^{k}(x+\alpha)^{\ell},\quad \mbox{whenever} \; \; k+\ell<n.
\label{6b}
\end{equation}
Let
\begin{equation}
W_n(\alpha,\beta)=
\begin{bmatrix}\mathcal J_\beta^{n-1}{\bf e}_1 & \mathcal J_\beta^{n-2}\mathcal J_\alpha{\bf e}_1&\ldots & 
\mathcal J_\beta\mathcal J_\alpha^{n-2}{\bf e}_1 &\mathcal J_\alpha^{n-1}{\bf e}_1\end{bmatrix}.
\label{7}
\end{equation}
Then $W_n(\alpha,\beta)$ is the unique matrix satisfying the identity
\begin{eqnarray}
&&\begin{bmatrix}1&x&\ldots&x^{n-1}\end{bmatrix}W(\alpha,\beta)\nonumber\\
&&=\begin{bmatrix}(x+\beta)^{n-1}&(x+\beta)^{n-2}(x+\alpha)&\ldots&(x+\alpha)^{n-1}\end{bmatrix}
\label{9}
\end{eqnarray}
for all $x$. Furthermore, $\det W_n(\alpha,\beta)=(\beta-\alpha)^\frac{n(n-1)}{2}$.
Finally, if $\alpha \ne \beta$, then $L = W_{2n-1}(\alpha,\beta)^\top$ is invertible 
and satisfies $L(\cR) \subseteq \cR$.
\end{lemma}

\begin{proof} 
We derive (\ref{6b}) from \eqref{6a}. 
Note that $\mathcal J_\beta=(\beta-\alpha)I_n+\mathcal J_\alpha$ and hence,
$$
\mathcal J_\beta^k=\sum_{j=0}^k \spm{k \\ j}(\beta-\alpha)^{k-j}\mathcal J_\alpha^j.
$$
Making use of the latter relation and \eqref{6a}, we get
\begin{eqnarray*}
&&\begin{bmatrix}1&x&\ldots&x^{n-1}\end{bmatrix}\mathcal J_\beta^{k}\mathcal J_\alpha^{\ell}{\bf e}_1\\
&&={\displaystyle\sum_{j=0}^k\spm{k \\ j}(\beta-\alpha)^{k-j}\begin{bmatrix}1&x&\ldots&x^{n-1}\end{bmatrix}\mathcal J_\alpha^{j+\ell}{\bf e}_1}\\
&&={\displaystyle\sum_{j=0}^k\spm{k \\ j}(\beta-\alpha)^{k-j}(x+\alpha)^{j+\ell}}\\
&&=(x+\alpha)^{\ell}\left(x+\alpha+\beta-\alpha\right)^k=(x+\beta)^{k}(x+\alpha)^{\ell}.
\end{eqnarray*}
Thus, (\ref{6b}) holds. Equality \eqref{9} is equivalent to equalities 
$$
\begin{bmatrix}1&x&\ldots&x^{n-1}\end{bmatrix}\mathcal J_\beta^{n-j}\mathcal J_\alpha^{j-1}{\bf e}_1
=(x+\beta)^{n-j}(x+\alpha)^{j-1}\quad \mbox{for} \; \;
j=1,\ldots,n, 
$$
which hold true by (\ref{6b}) and uniquely determine 
the entries of the columns $\mathcal J_\beta^{n-j}\mathcal J_\alpha^{j-1}{\bf e}_1$ in $W_n(\alpha,\beta)$.

\smallskip

It is clear from \eqref{3} that $\det V_n(\alpha)=1$. 
To evaluate the determinant of $W_n(\alpha,\beta)$, we first observe from \eqref{7} that 
$$
W_n(0,\beta)=\begin{bmatrix}\mathcal J_\beta^{n-1}{\bf e}_1 & Z_n\mathcal J_\beta^{n-2}{\bf e}_1&\ldots &
Z_n^{n-2}\mathcal J_\beta{\bf e}_1 &Z_n^{n-1}{\bf e}_1\end{bmatrix}
$$
is a lower triangular matrix with diagonal entries
$$
\left[W_n(0,\beta)\right]_{jj}={\bf e}_j^\top Z^{j-1}_n\mathcal J_\beta^{n-j}{\bf e}_1={\bf e}_1^\top\mathcal J_\beta^{n-j}{\bf e}_1=\beta^{n-j},
\quad j=1,\ldots n,
$$
and therefore, $\det W_n(0,\beta)=\beta^\frac{n(n-1)}{2}$. By \eqref{9}, we have
$$
\begin{bmatrix}1&x&\ldots&x^{n-1}\end{bmatrix}W_n(0,\beta)
=\begin{bmatrix}(x+\beta)^{n-1}&(x+\beta)^{n-2}x&\ldots&x^{n-1}\end{bmatrix}.
$$
Replacing $x$ and $\beta$ in the latter identity by respectively, $x+\alpha$ and $\beta-\alpha$, gives
\begin{eqnarray*}
&&\begin{bmatrix}1&x+\alpha&\ldots&(x+\alpha)^{n-1}\end{bmatrix}W(0,\beta-\alpha)\\
&&=\begin{bmatrix}(x+\beta)^{n-1}&(x+\beta)^{n-2}(x+\alpha)&\ldots&(x+\alpha)^{n-1}\end{bmatrix},
\end{eqnarray*}
which being combined with \eqref{6} and \eqref{9}leads us to 
$$
\begin{bmatrix}1&x&\ldots&x^{n-1}\end{bmatrix}V_n(\alpha)W_n(0,\beta-\alpha)=\begin{bmatrix}1&x&\ldots&x^{n-1}\end{bmatrix}W_n(\alpha,\beta).
$$
By the uniqueness property in \eqref{9}, 
$V_n(\alpha)W_n(0,\beta-\alpha)=W_n(\alpha,\beta)$ and hence,
\begin{eqnarray*}
\det W_n(\alpha,\beta)&=&\det V_n(\alpha)\cdot \det W(0,\beta-\alpha)\\
&=&\det W_n(0,\beta-\alpha)=(\beta-\alpha)^\frac{n(n-1)}{2},
\end{eqnarray*}
which completes the computation. 

\smallskip

Finally, let $L = W_{2n-1}(\alpha,\beta)^\top$. If $\xi \in \IF \setminus \{-\beta \}$, 
then 
$$
L(h(\xi)) = \mu h(y)\quad\mbox{with} \; \;  \mu = (x+\beta)^{n-1}, \; y = \frac{x+\alpha}{x+\beta}.
$$
Furthermore, 
$$
L(h(-\beta)) = (x+\alpha)^{n+1} h(\infty),\quad L(h(\infty)) = [1, \alpha, \dots, \alpha^{2n-1}]^\top = h(\alpha).
$$
Thus, $L = W_{2n-1}(\alpha,\beta)^\top$ preserves $\cR$.
\end{proof}

The next lemma shows that $L \in M_{2n-1}$ satisfying $L(\cR) \subseteq \cR$ is either invertible 
or a rank one matrix in the real case. Once it is proved,  we can present the main theorem
giving the complete description of matrices in $M_{2n-1}$ preserving $\cR$.

\begin{lemma} \label{lem2}
Let $L \in M_{2n-1}$ satisfy $L(\cR) \subseteq \cR$.
Then  either
\begin{enumerate}
\item[{\rm (a)}] $L$ is a invertible, or
\item[{\rm (b)}] $\IF = \IR$ and $L = h[c_1, \dots, c_{2n-1}]$
for some $h\in \cR$ and  $c_1, \dots, c_{2n-1} \in \IR$ such that 
$$c_1 + c_2 x + \dots + c_{2n-1}x^{2n-2} \ne 0  \hbox{ for any } x \in \IR.$$
\end{enumerate}
\end{lemma}
\begin{proof}  If $L$ is invertible, then (a) holds. Assume that $L$ is singular.
Let us choose distinct $\xi_1,\ldots,\xi_k$ for $k=(2n-2)^2+1$, introduce the vectors
\begin{equation}
y_j=Lh(\xi_j)\in\mathcal R\quad\mbox{for}\quad j=1,\ldots,k,
\label{sety}
\end{equation}
and define an equivalence relation on the set \eqref{sety}: {\em $y_i\sim y_j$ if and only if $y_i$ and $y_j$ are 
multiples of each other}. Since $L$ is singular, any collection of $2n-1$ vectors from \eqref{sety} is linearly dependent. 
By Proposition \ref{R:1.3},  any collection of $2n-1$ vectors from \eqref{sety} contains two equivalent
vectors, i.e., the two vectors that are multiples of each other. Therefore, there are at most $2n-2$ equivalence classes, and one of them must contain at least $2n-1$ elements,
say, $y_1, \dots, y_{2n-1}$. Then the matrix $Y=\begin{bmatrix}y_1&\ldots&y_{2n-1}\end{bmatrix}$ has rank one. Since
$$
Y=\begin{bmatrix}Lh(\xi_1)&\ldots&Lh(\xi_{2n-1})\end{bmatrix}=
L\begin{bmatrix}h(\xi_1)&\ldots&h(\xi_{2n-1})\end{bmatrix}
$$ 
where $\begin{bmatrix}h(\xi_1)&\ldots&h(\xi_{2n-1})\end{bmatrix}$ is an invertible Vandermonde matrix, we conclude
that ${\rm rank} \, L=1$. It remains to show that the latter cannot happen if $\mathbb F=\mathbb C$ and may occur if 
$\mathbb F=\mathbb R$. To this end, we write $L$ of rank one in the form 
\begin{equation}
L={\bf b} {\bf p}, \; \; \mbox{where}\; \;  {\bf b}=\sbm{b_1 \vspace{-1mm}\\ \vdots \\ b_{2n-1}}\neq 0,
\quad {\bf p}=\begin{bmatrix}p_0 & \ldots & p_{2n-2}\end{bmatrix}\neq 0.
\label{2x}
\end{equation}
Since the vectors $Lh(0)={\bf b}p_0$ and $Lh(\infty)={\bf b}p_{2n-2}$ belong to $\mathcal R$ and in particular are nonzero,  
we have $p_0\neq 0$, $p_{2n-2}\neq 0$, and ${\bf b}\in\mathcal R$. Upon rescaling ${\bf p}$ by a nonzero $\mu$ (if necessary)
we conclude that either
\begin{equation}
L=h(\infty){\bf p}\quad\mbox{or}\quad L=h(\zeta){\bf p}\quad\mbox{for some} \; \; \zeta\in\mathbb F.
\label{2y}
\end{equation}
Besides, $Lh(\xi)\in\mathcal R$ and hence $Lh(\xi)\neq 0$ for each $\xi\in\mathbb F$, so that 
\begin{equation}
{\bf p}h(\xi)=p_0+p_1\xi+\ldots+p_{2n-2}\xi^{2n-2}:=p(\xi)\neq 0\quad\mbox{for all} \; \; \xi\in\mathbb F,
\label{2z}
\end{equation}
where $p(x)$ is a polynomial of degree $\deg p=2n-2$. If $\mathbb F=\mathbb C$, we get a contradiction, since
any complex polynomial of positive degree has a root in $\mathbb C$ (we exclude the trivial case $n=1$). 
If $\mathbb F=\mathbb C$, we do have rank-one matrices $L$ of the form \eqref{2y} with ${\bf p}$ subject to condition \eqref{2z}.
\end{proof}

We are now ready to present the characterization of
$L \in M_{2n-1}$ satisfying $L(\cR) \subseteq \cR$, where $\cR$ is defined as in 
(\ref{calr}). We will use matrices $F_{2n-1}$, $D_{2n-1}(r), V_{2n-1}(\alpha), W_{2n-1}(\alpha,\beta)$ introduced in (\ref{a}),
(\ref{4}) and (\ref{7}). 

\begin{theorem} \label{main1} A matrix $L \in M_{2n-1}$ satisfies $L(\cR) \subseteq \cR$
if and only if one of the following holds:
\begin{enumerate} 
\item[{\rm (a)}] $L^\top$ is of one of the following forms: 
\begin{itemize}
\item{\rm (i)} \hspace{1mm} $L^\top=\gamma V_{2n-1}(\alpha) D_{2n-1}(r)$, 
\item{\rm (ii)} \, $L^\top=\gamma V_{2n-1}(\alpha)F_{2n-1}D_{2n-1}(r)$, 
\item{\rm (iii)}\hspace{0.4mm}  $L^\top=\gamma  W_{2n-1}(\alpha,\beta)D_{2n-1}(r)$,
\end{itemize}
for some $\gamma, r, \alpha, \beta\in \IF$ with $\gamma r \ne 0$ and $\alpha \ne \beta$. 

\item[{\rm (b)}] $\IF = \IR$ and $L = h[c_1, \dots, c_{2n-1}]$
for an $h\in \cR$ and  $c_1, \dots, c_{2n-1} \in \IR$ such that 
$$
c_1 + c_2 x + \dots + c_{2n-1}x^{2n-2} \ne 0  \quad\mbox{for any} \; \;  x \in \IR.
$$ 
\end{enumerate}
\end{theorem}

\begin{proof}  If (a) holds, then $L(\cR) \subseteq \cR$, by Lemmas \ref{L:2.1} and \ref{L:2.2}.
If (b) holds, then clearly $L(\cR) \subseteq \cR$.

\smallskip

We focus on the necessity. By Lemma \ref{lem2}, if $L$ is singular, then (b) holds.
So, we focus on the case when $L$ is invertible.
With $L=\left[p_{ij}\right]_{i=1,\ldots,2n-1}^{j=0,\ldots,2n-2}$ we associate the polynomials 
\begin{equation}
P_j(x)= p_{j0} + p_{j1} x + \cdots + p_{j,2n-2}x^{2n-2}\quad\mbox{for}\quad j = 1, \dots, 2n-1.
\label{Pj}
\end{equation}
Then
\begin{equation}
Lh(x)=\sbm{P_1(x)\\ P_2(x)\vspace{-1mm}\\ \vdots \\ P_{2n-1}(x)}\quad\mbox{for all} \; \; x\in\mathbb F \; \; \mbox{and}\; \; 
h(x)=\sbm{1\\ x\vspace{-1mm}\\ \vdots \\ x^{2n-2}}.
\label{Pj1}
\end{equation}
It is clear that $P_j(x)$ completely determines the $j$-th row of $L$. 
We will show that one of the following three cases holds.
\begin{eqnarray}
{\rm (a)} \; P_{j}(x)&=& \gamma r^{j-1}(x+\alpha)^{j-1},\nonumber\\
{\rm (b)} \; P_{j}(x) &=& \gamma r^{j-1}(x+\alpha)^{2n-j-1},\label{ps}\\
{\rm (c)} \; P_{j}(x) &=& \gamma r^{j-1}(x+\alpha)^{j-1}(x+\beta)^{2n-j-1},\quad j=1,\ldots, 2n-1\nonumber
\end{eqnarray}
for some $r,\gamma,\alpha, \beta\in\mathbb F$ such that $r,\gamma\neq 0$ and $\alpha\neq \beta$. 
Then $L$ will be in the form in (a).

\medskip
Suppose $L\in M_{2n-1}$ is invertible and $L(\cR) \subseteq \cR$.
Since a nonzero vector $h = (h_1, \dots, h_{2n-1})^\top$ belongs to $\cR$ if and only if 
$h_{j-1}h_{j+1} = h_{j}^2$ for $j = 2, \dots, 2n-2$ (see \eqref{calr-2})  we have
\begin{equation}
P_{j-1}(x) P_{j+1}(x) = P_{j}(x)^2 \; \; \mbox{for} \; \; j = 2, \dots, 2n-2. 
\label{eq}
\end{equation}
Since $L$ is invertible, $P_j\not\equiv 0$ and we let $\gamma_j\neq 0$ to denote its leading coefficient.
Then $\widetilde{P}_j(x) = P_j(x)/\gamma_j$ is a monic polynomial, and we have by \eqref{eq},         
$$
\gamma_j \gamma_{j+2} = \gamma_{j+1}^2 \quad \hbox{ and } \quad
\widetilde{P}_{j-1}(x) \widetilde{P}_{j+1}(x) = \widetilde{P}_{j}(x)^2 \quad \mbox{for} \; \; j = 2, \dots, 2n-2.
$$
By the first relation, the nonzero sequence $\gamma_1, \dots, \gamma_{2n-1}$ is geometric and hence,
$$
\gamma_j=\gamma r^{j-1}\quad\mbox{for} \; \; j=1,\ldots.2n-1,
$$
where $\gamma=\gamma_1$ and $r=\gamma_2/\gamma_1$. Next, we turn to $\widetilde{P}_1, \dots, \widetilde{P}_{2n-1}$.
Since $L$ is invertible, there exists $k$ such that $\deg P_k=\deg\widetilde{P}_k>0$.
Then $\widetilde{P}_k$ has a (complex) root $x=-\alpha$. Let us factor the polynomials $\widetilde{P}_j$ as 
\begin{equation}
\widetilde{P}_j(x)=(x+\alpha)^{\ell_j}Q_j(x), \quad \ell_j\ge 0, \; \; j=1,\ldots,2n-1,
\label{eqa}
\end{equation}
where $Q_j$ is a monic polynomial such that $Q_j(-\alpha)\neq 0$. By our assumption, $\ell_k>0$. Also, 
$\ell_j\le 2n-2$ since $\deg \widetilde{P}_j\le 2n-2$ for all $j$. Substituting the latter 
factorizations into \eqref{eq}, we get
$$
(x+\alpha)^{\ell_{j-1}+\ell_{j+1}}Q_{j-1}(x) Q_{j+1}(x)=(x+\alpha)^{2\ell_{j}}Q_{j}(x)^2.
$$
As a result, for $j=1,\ldots, 2n-2$, 
\begin{equation}
\ell_{j-1}+\ell_{j+1}=2\ell_{j}\quad\mbox{and}\quad
Q_{j-1}(x) Q_{j+1}(x)=Q_{j}(x)^2  .
\label{eqb}
\end{equation}
Therefore, $\boldsymbol\ell=\{\ell_1,\ldots,\ell_{2n-1}\}$ is a nonnegative integer arithmetic sequence bounded above by $2n-2$ 
and with $\ell_k\ge 1$.
If $\ell_j=\ell_k$ for all $j$, then $P_j(-\alpha)=0$ for all $j$ and hence $Lh(-\alpha)=0$, which contradicts the invertibility
of $L$. If the sequence $\boldsymbol\ell$ increases or decreases, then it contains $2n-1$ distinct integers bounded by zero and 
$2n-2$. Hence, either $\ell_j=j-1$ or $\ell_j=2n-j-1$ for all $j=1,\ldots,2n-1$. In view of \eqref{eqa}, we conclude that either
\begin{equation}
\widetilde{P}_j(x)=(x+\alpha)^{j-1}Q_j(x)\quad\mbox{or}\quad \widetilde{P}_j(x)=(x+\alpha)^{2n-j-1}Q_j(x)
\label{eqc}
\end{equation}
for all $j=1,\ldots,2n-1$. In the first case, we have $Q_{2n-1}\equiv 1$, since $Q_{2n-1}$ is monic and 
$\deg \widetilde{P}_{2n-1}$ cannot exceed $2n-2$. For the same reason $\deg Q_{2n-2}\le 1$ and hence either
$Q_{2n-2}\equiv 1$ or $Q_{2n-2}(x)=x+\beta$ for some $\beta\in\mathbb F$. Then we recursively get from \eqref{eqb} that either 
$$
Q_j\equiv 1\quad\mbox{or}\quad Q_j(x)=(x+\beta)^{2n-j-1}\quad\mbox{for all}\; \;  j=1,\ldots,2n-1.
$$
Combining the latter formulas with the first formula in \eqref{eqc} and recalling that $P_j(x)=\gamma r^{j-1}\widetilde{P}_j(x)$
we get cases (a) and (c) in\eqref{ps}. In case $\widetilde{P}_j$ is given by the second formula in \eqref{eqc}, we use the same 
degree argument to conclude that 
either
$$
Q_j\equiv 1\quad\mbox{or}\quad Q_j(x)=(x+\beta)^{j-1}\quad\mbox{for all}\; \;  j=1,\ldots,2n-1.
$$
Combining the latter with the second formula in \eqref{eqc} we come to polynomials $P_j$ as in part (b) in \eqref{ps} or
as in part (c) with $\alpha$ and $\beta$ switched. \end{proof}

\smallskip

We may use Theorem \ref{main1} to describe the rank one preservers on $\cT_n$ in a more traditional form \eqref{class}.
We first observe that any linear functional $f: \, \mathcal T_n\to \mathbb R$ is of the form 
\begin{equation}
f_{\bf c}: \;A=\left[a_{i-j}\right]_{i,j=1}^n\mapsto \sum_{j=1}^{2n-1}c_ja_{n-j},\quad {\bf c}=(c_1,\ldots,c_{2n-1})\in\mathbb R^{2n-1}.
\label{adf}
\end{equation}
Let us say that $f_{\bf c}$ is {\em admissible} if 
$$
p(x):=\sum_{j=0}^{2n-2}c_{j+1}x^j\neq 0\quad\mbox{for all} \; \; x\in\mathbb R.
$$
The latter means that for any matrix 
\begin{equation}
X_\xi = \sbm{\xi^{n-1}\\ \vdots \\ \xi \\ 1}\sbm{1 & \xi & \ldots& \xi^{n-1}}\quad\mbox{with} \; \;
\xi \in \IR,
\label{xxi}
\end{equation}
we have $f(X_\xi) = {\displaystyle \sum_{j=0}^{2n-2} c_{j+1} \xi^j} \ne 0$.
\begin{theorem} \label{main-man}
Let $T: {\cT}_n\rightarrow {\cT}_n$ be a linear map.
Then $T$ preserves rank one matrices in ${\cT}_n$  if and only if one of the following holds.
\begin{itemize}
\item[{\rm (a)}] There are nonzero $\gamma, r \in \IF$ such that $T$ has the form 
$$A \mapsto MAN$$
with $M = \gamma F_n N^\top F_n$ and $N$ having one of the following 
form:
$${\rm (1)} \; N=V_n(\alpha)D_n(r),
\quad {\rm (2)} \;N=V_n(\alpha)F_nD_n(r),\quad 
{\rm (3)} \;  N=W_n(\alpha,\beta)D_n(r),
$$
where $D_n(r)$  and $F_n$ are defined in \eqref{a}.

\item[{\rm (b)}] $\IF = \IR$
and $T$ has the form $A \mapsto f_{\bf c}(A)B$ for some rank one matrix $B \in \cT_n$
and an admissible linear functional $f_{\bf c}$ on $\cT_n$.
 \end{itemize}
\end{theorem}

\begin{proof} The set $\hat \cR$ of matrices in $\cT_n$ of the form \eqref{xxi}
contains a generating set of $\cT_n$, and the structure of a linear 
map $T: \cT_n\rightarrow \cT_n$ is completely determined by its action on 
the matrices in $\hat \cR$.  
It is clear that a linear map $T: \cT_n\rightarrow \cT_n$ preserves rank one matrices if and only if 
the matrix $L = [T]_\cB$ satisfies conditions (a) or (b) in Theorem \ref{main1}.

\smallskip

If $L$ satisfies Theorem \ref{main1} (b), then one readily
checks that condition (b) holds.

Suppose $L = [T]_\cB$ is invertible.  It is easily verified that 

\begin{itemize}
\item if $L = \gamma D_{2n-1}(r)$, then the action of $T$ on $X_\xi$ has the form  
$$
X_\xi \mapsto \gamma F_nD_n(r)F_n X_\xi D_n(r);
$$
\item if $L = V_{2n-1}(\alpha)$, then 
the action of $T$ on $X_\xi$ has the form  
$$
X_\xi \mapsto \gamma F_n V_n(\alpha)^\top F_n X_\xi V_n(\alpha);
$$

\item if $L = W_{2n-1}(\alpha,\beta)$, then the action of $T$ 
on $X_\xi$ has  the form 
$$
X_\xi \mapsto \sbm{(\xi+\alpha)^{n-1} \\ (\xi+\alpha)^{n-2}(\xi+\beta)
\vspace{-1mm} \\ \vdots \\ (\xi+\alpha)(\xi+\beta)^{n-2} \\ (\xi+\beta)^{n-1}}\sbm{(\xi+\beta)^{n-1}&
(\xi+\beta)^{n-2}(\xi+\alpha)&\ldots&(\xi+\beta)(\xi+\alpha)^{n-2}&(\xi+\alpha)^{n-1}},
$$
that is, $X_\xi \mapsto F_n W(\alpha,\beta)^\top F_n X_\xi W(\alpha,\beta)$.
\end{itemize}
Together with the fact that $F_n F_n = I_n$, one can check that
 Theorem \ref{main1} (a) holds if and only if
 condition (a) in the theorem holds, with $N$ having the three
 forms according to (i), (ii), (iii) in Theorem \ref{main1} holds.
 \end{proof}

In the classical result, rank one preservers has the form 
$A \mapsto RAS$ or $A \mapsto RA^\top S$ for some invertible matrices $R, S$.
In our case, when $\IF = \IR$ there are rank one preserver $T$ such that
the range space has dimension 1. In all other cases, 
$T$ has the form $A \mapsto RAS$ for some invertible $R$ and $S$ with more 
restrictions, which can be determined by at most four parameters $\gamma, r, \alpha, \beta$.
Also, 
note that in $\cT_n$, the map $A \mapsto A^\top$ can be written as $A \mapsto F_nAF_n$. 
So, the mappings of the form $A\mapsto RA^\top S$ will reduce to mappings of the form
$A \mapsto \tilde R A \tilde S$ by suitable adjustment of $R$ and $S$.

\section{Related problems and concluding remarks}

\subsection{Rank preservers and determinant preservers}

One may consider linear maps $T$ on $\cT_n$ such that 
$T(\hat \cR) = \hat \cR$, where $\hat \cR$ is the set of
rank one matrices in $\cT_n$.
In such a case, we say that $T$ {\bf strongly preserves} of
$\hat\cR$.
Since $\hat \cR$ contains a spanning set of $\IF^{2n-1}$, 
such a map will map a spanning set
onto a spanning set of $\cT_n$, and must be invertible.
Hence, $T$ is invertible and preserves $\hat \cR$.
So, $T$ has the form as in Theorem \ref{main-man} (a).
We may consider other preserving conditions that lead to the 
same structure.

\begin{theorem} Suppose $T: \cT_n\rightarrow \cT_n$ with 
$[T]_\cB = L \in M_{2n-1}$. The following conditions are equivalent.
\begin{itemize}
\item[{\rm (1)}] $L$ satisfies Theorem {\rm \ref{main1} (a)}, i.e., 
$T$ satisfies Theorem {\rm \ref{main-man} (a)}.
\item[{\rm (2)}] $L$ is invertible and $L^{-1}(\cR)$, i.e.,
$T$ is invertible and $T^{-1}$ maps the set of rank one matrices in $\cT_n$ into
itself. 
\item[{\rm (3)}] $L(\cR) = \cR$, i.e.,
$T$ maps the set of rank one matrices $\cT_n$ onto itself. 
\end{itemize}
\end{theorem}

\begin{proof}
Suppose (1) holds. Then $L$ is invertible and  
$L(\cR) \subseteq \cR$ by Theorem \ref{main1}.  
For any distinct elements
$\xi_1, \dots, \xi_N \in \IF$, we have
$L[h(\xi_1) \cdots h(\xi_k)] = [y_1 \cdots y_k]$.
If $i\ne j$, then  $y_i$ and $y_j$ cannot be multiples of each other.
Note that at most one $y_\ell$ is a multiple of $h(\infty)$. If it happens, we may 
replace $\xi_\ell$ by a different value, and assume that 
$y_j = \mu_j h(\zeta_j)$ with $\zeta_j\in \IF$ for $j = 1, \dots, k$.
Then 
$$L^{-1} [h(\zeta_1) \cdots h(\zeta_{k-1})] = 
[h(\xi_1)/\mu_1 \cdots h(\xi_{k-1})/\mu_{k-1}].$$
Suppose $k-1 \ge 4n-2$.
We may now assume that the rows of $L^{-1}$ correspond 
to polynomials $P_1(x), \dots, P_{2n-1}(x)$ such that
$$P_j(x) P_{j+2}(x) = P_j(x)^2 \quad \mbox{for} \; \;  x = \zeta_1, \dots, \zeta_{4n-2} \in \IF 
$$
with $\zeta_i \ne \zeta_j$, whenever $i\ne j$.
Using the arguments in the proof of Theorem \ref{main1}, we see that 
$L^{-1}$ has the form in (\ref{ps}), and $L^{-1}(\cR) \subseteq \cR$. 
Thus, condition (2) holds. 

\smallskip

Now, suppose condition (2) holds. 
Switch the roles of $L$ and $L^{-1}$ and apply 
the argument from the preceding paragraph.
We see that $L(\cR) \subseteq \cR$. Thus, $L(\cR) = \cR$, i.e., condition (3) holds.

Suppose (3) holds. By the discussion before the theorem,
we see that (1) holds.
\end{proof}

One may consider linear preservers on $\cT_n$ preserving the rank of matrices,
i.e., $T(A)$ and $A$ always have the same rank. By Theorem \ref{main-man}, we have the following.

\begin{theorem} A linear map $T: \cT_n\rightarrow \cT_n$ preserves the rank of matrices
if and only if $T$ has the form as in Theorem {\rm\ref{main-man} (a)}.
\end{theorem}

\begin{proof} If $T$ preserves rank, it will preserve rank one matrices.
By Theorem \ref{main-man}, it has the form (a) or (b). Evidently, it cannot be of the form
(b). Otherwise, $T(A)$ has rank at most one for all $A$. Thus, it has the form as in (a).
The converse is clear. 
\end{proof}

\begin{theorem} \label{thm3} A linear map $T:{\cT}_n\rightarrow {\cT}_n$ satisfies $\det(T(A)) = \det(A)$ for all $A \in {\cT}_n$ 
if and only if $T$ has the form as in Theorem {\rm \ref{main-man} (a)} 
and 
$$1 =  \det(\gamma F_n D_n(r)N^\top F_n D_n(r)),$$ 
equivalently, 
$$\gamma^n r^{n(n-1)} = 1 \ \ \hbox{ and } \ \  
(\alpha-\beta)^{n(n-1)} = 1 \ \ \hbox{ in case } \ \ N = W(\alpha,\beta).$$
\end{theorem}

\begin{proof}
Assume that  $T:{\cT}_n\rightarrow {\cT}_n$ is linear and preserves
the determinant.
We  establish three assertions to prove the necessity.

\medskip\noindent
{\bf Claim 1.}  $T$ is bijective. 

If not, there is $A \in {\cT}_n$ such that $T(A) = 0$. Then we 
can find a singular $B \in {\cT}_n$ such that
$A+B$ is invertible. But then $0 \ne \det(A+B) = \det(T(A+B)) = \det(T(B)) = 0$,
which is a contradiction.

\medskip\noindent
{\bf Claim 2.}  $T$ preserves rank one matrices in  $\cT_n$.

Note 
that a matrix $A \in {\cT}_n$ is rank one if and only if there are invertible matrices
$R, S$ such that $A = R\be_1 \be_1^\top S$. 
Thus, $A$ has rank one if and only if $\det(xA + B) 
= \det(R)\det(x\be_1 \be_1^\top + R^{-1}BS^{-1}) \det(S)
= ax + b_B$ for all $B \in {\cT}_n$. Since $T$ is bijective and 
preserves determinant, 
$\det(xT(A) + T(B)) = ax + b$ for all matrices $B$. So, $T(A)$ has rank one.

\medskip\noindent
{\bf Claim 3} The mapping $T$ has the form as in Theorem \ref{main-man} (a) with 
$1 =  \det(\gamma F_n D_n(r)N^\top F_n D_n(r))$; equivalently, 
$\gamma^n r^{n(n-1)} = 1$ and $(\alpha-\beta)^{n(n-1)} = 1$ 
in the case $N = W(\alpha,\beta)$.

By Claim 2, $T$ preserves rank one matrices so that $T$ has the form in  
Theorem \ref{main-man} (a).  Moreover, for any $A \in \cT_n$, 
\begin{eqnarray*}
\det(A) & = & \det(T(A)) = 
\det(\gamma F_n D_n(r) N^\top F_n)\det(A) \det(N D_n(r))\\
& = & \det(A) \det(\gamma I_n)\det(F_n)^2 \det(D_n(r)^2)
\det(N^\top N) \\
& = & \det(A) \gamma^n r^{n(n-1)} \det(N^\top N).
\end{eqnarray*}
If $N = V_n(\alpha)$, then $\det(N^\top N) = 1$.
If $N = W(\alpha,\beta)$,  then 
$\det(N^\top N) = (\alpha-\beta)^{n(n-1)}$
by Lemma \ref{L:2.2}.
The converse is clear. 
\end{proof}

\subsection{Maps from Toeplitz to general matrices}

One may consider linear maps $T: {\cT}_n\rightarrow M_n$. Such preserver problems arise in applications when one
can apply linear transformations to Toeplitz matrices, but have no guarantee that the
transformed matrices will be Toeplitz. For example, linear isometries $T: {\cT}_n\rightarrow M_n$
were characterized in \cite{Fa}. For linear rank one preservers $T: {\cT}_n\rightarrow M_n$,
we have the following.

\begin{theorem} Suppose $n > 2$ and  $T: {\cT}_n \rightarrow M_n$ is a linear map.
Then $T$ sends rank one matrices to rank one matrices if and only if 
$T$ has the form $A \mapsto MAN$ such that
$M, N\in M_n$ satisfy the following:
for any $x\in \IF$
$$M[x^{n-1}, \dots, x, 1]^\top = u(x) 
= [u_1(x), \dots, u_n(x)]^\top \ne [0,\dots,0]^\top,$$
$$[1, x, \dots, x^{n-1}]N = v(x)^\top = [v_1(x), \dots,  v_n(x)] \ne [0,\dots,0],$$
where $u_1(x), \dots, u_n(x)$ are polynomials in $x$ with coefficients associated with the 
rows of $M$ and $v_1(x), \dots, v_n(x)$ are polynomials in $x$ with coefficients 
associated with the columns of $N$.
\end{theorem}

\it Proof. \rm  The sufficiency is clear. We prove the necessity.

Suppose $L$ is the matrix of transformation of $T$ with respect to the basis $\cB$ of $\mathcal T_n$ as in \eqref{basis}
and the standard basis for $M_n$.
Then $L\in\mathbb M_{n^2,(2n-1)}$, the set of $n^2 \times (2n-1)$ matrices over $\IF$.  
Consider rank one matrix in ${\cT}_n$ of the form $X_x$ \eqref{xxi}. The vector of $X$ with respect to $\cB$
has the form $[1, x, \dots, x^{2n-2}]^\top$, and 
$L(X) = [P_1(x),\dots, P_{n^2}(x)]^\top$
for some polynomials $P_1(x), \dots, P_{n^2}(x)$, each has degree at most $2n-2$.
It follows that 
$$
T(X_x) = A_0 + A_1x + \cdots + A_{2n-2}x^{2n-2},$$
for some matrices $A_0, \dots, A_{2n-1}$.
Since $T(X)$ is always rank one, it follows that $T(X) = u(x) v(x)^\top$ for some 
$u(x) = [u_1(x), \dots, u_n(x)]^\top$ and $v(x) = [v_1(x), \dots, v_n(x)]^\top$;
see \cite{GLR}. The result follows.
\qed

Note that the map $T$ in the theorem is injective if and only if the matrix of transformation
$L$ has rank $2n-1$. The map $T$ preserves determinant if and only if $\det(MN) = 1$.

\subsection{Maps on rectangular Toeplitz matrices}

Let $2 \le m \le n$, and let $\cT_{m,n}$ be the set of 
$m \times n$ Toeplitz matrices, that is, matrices of the form
\[
A = \begin{bmatrix} 
a_0 & a_{-1} & a_{-2} & \dots & \dots & \dots & a_{-(n-1)} \\
a_1 & a_0 & a_{-1} & \dots & \dots & \dots & a_{-(n-2)} \\
a_2 & a_1 & a_0 & \dots & \dots & \dots & a_{-(n-3)} \\
\vdots & \vdots & \vdots & \ddots & \vdots & \vdots &\vdots \\
a_{m-1} & a_{m-2} & a_{m-3} & \dots & a_0 & \dots & a_{n-m}
\end{bmatrix} . \]
Then $\cT_{m,n}$ has dimension $m+n-1$ with a basis consisting of 
matrices of the form
\begin{equation}
X_\xi=\begin{bmatrix} \xi^{m-1} \cr \vdots \cr \xi \cr 1\cr\end{bmatrix}
[1 \ \xi \ \dots \xi^{n-1} ], \qquad \xi \in \IF.
\label{class1}
\end{equation}
In fact, if one chooses any $m+n-1$ distinct values in $\IF$ then 
one gets a basis for $\cT_{m,n}$.  
A basis $\tilde \cB$ may be obtained by
removing the last $n-m$ rows of the matrices in $\cB$ and then remove 
the zero matrices.

\smallskip

Linear maps $T: \cT_{m,n}\to \cT_{m,n}$ may be studied by considering 
$L = [T]_{\tilde \cB}$. The proofs from the previous sections go through.
For example, Theorem \ref{main1} and Theorem \ref{main-man} become:

\begin{theorem} Let $T: \cT_{m,n} \rightarrow \cT_{m,n}$ be a linear map with $2 \le m \le n$
and let $L:=[T]_{\tilde \cB} \in M_{m+n-1}$. Then the set of rank one matrices 
in $\cT_{m,n}$ has coordinate vectors in the set 
$$\tilde \cR = \{[1, \xi, \dots, \xi^{m+n-2}]^\top: \xi \in \IF \cup \{\infty\}\}.$$
The linear function $T$ 
maps the set of rank one matrices in $\cT_{m,n}$ into itself,
i.e., $L(\tilde \cR) \subseteq \tilde \cR$ if and only if one of the following conditions
holds.
\begin{itemize}
\item[{\rm (a)}] There are nonzero $\gamma, r \in \IF$ such that $T$ has the form 
$$
A\mapsto \gamma F_mD_m(r)N_1^\top F_m A N_2 D_n(r), 
$$
where for $k = m, n$, the matrices $D_k(r)$  and $F_k$ are defined in \eqref{a} and 
$N_1$ and $N_2$ satisfy one of the following: 
\begin{itemize}
\item[{\rm (1)}] $N_1=V_m(\alpha)^\top,\quad N_2=V_n(\alpha)$;
\item[{\rm (2)}] $N_1=(F_m V_m(\alpha)^\top,\quad N_2=V_n(\alpha)F_n$;
\item[{\rm (3)}] $N_1=W_m(\alpha,\beta)^\top,\quad N_2=W_n(\alpha,\beta)$.
\end{itemize}

\smallskip\noindent
Equivalently, for $\ell = m+n-1$, $L^\top$ equals to 

\smallskip

{\rm (i)} $\gamma V_{\ell}(\alpha) D_{\ell}(r)$, \ \
{\rm (ii)} $\gamma V_{\ell}(\alpha)F_{\ell}D_{\ell}(r)$, \ \ \hbox{ or } \ \ 
{\rm (iii)} $\gamma  W_{\ell}(\alpha,\beta)D_{\ell}(r)$.

\smallskip

\item[{\rm (b)}] $\IF = \IR$
and $T$ has the form $A \mapsto f_{\bf c}(A)B$ for some rank one matrix $B \in \cT_n$
and an admissible linear functional $f_{\bf c}$ on $\cT_{m,n}$, i.e., a linear functional 
$$f_{\bf c}: \; A=\left[a_{i-j}\right]_{i=1,\ldots,n}^{j=1,\ldots,m}\mapsto \sum_{j=1}^{m+n-1}c_ja_{m-j},\quad {\bf c}=(c_1,\ldots,c_{m+n-1})
$$
such that the polynomial $p(x):=\sum_{j=0}^{m+n-2}c_{j+1}x^j$ has no real zeros; equivalently, for any matrix $X_\xi$ of the form 
\eqref{class1}, we have $f_{\bf c}(X_\xi)\ne 0$.
\end{itemize}
\end{theorem}
\subsection{Hankel matrices}
The results on linear preservers on Toeplitz matrices are easily translated to Hankel matrices, matrices in which each ascending skew-diagonal from left to right is constant; that is, 
matrices of the form $A=[a_{i+j}]$ for some $a_2, \dots, a_{2n}$. 
Evidently, $A$ is Hankel if and only if $F_nA$ is Toeplitz.
Thus, $T: \cH_n\rightarrow \cH_n$ is a linear map  if and only if
$\tilde T$ defined by $A \mapsto F_nT(F_nA)$ is a linear map on ${\cT}_n$.
In particular, $T$ is a rank one preserves 
on $\cH_n$ if and only $\tilde T$ is a rank one preserves
on ${\cT}_n$; $T$ is a determinant preserver on $\cH_n$ if and only
if  $\tilde T$ is a determinant preserver on ${\cT}_n$. So, 
we can translate the preserver results on ${\cT}_n$ in the previous sections 
to results on $\cH_n$.
For instance, we have the following.

\begin{theorem} Suppose $n > 2$ and
$\tilde T: {\cH}_n \rightarrow \cH_n$ is a linear map.
Then $\tilde T$ sends rank one matrices to rank one matrices
if and only if $\tilde T$ has the form
$$A \mapsto \gamma N^\top A N,$$
where $\gamma \in \IF$ is nonzero and $N$
has one of the following forms:
$${\rm (1)} \; N=V_n(\alpha)D_n(r),
\quad {\rm (2)} \;N=V_n(\alpha)F_nD_n(r),\quad
{\rm (3)} \;  N=W_n(\alpha,\beta)D_n(r),
$$
where $D_n(r)$  and $F_n$ are defined in \eqref{a}.

\smallskip\noindent
The  map $\tilde T$
preserves the determinant if and only it has the above form with $\det(\gamma N^\top N) = 1$, equivalently,
$\gamma^n r^{(n(n-1)} = 1$ and
$(\alpha-\beta)^{n(n-1)} = 1$ in case $N = W(\alpha,\beta)$.
\end{theorem}

\subsection{Concluding remarks}
A close examination of our proofs reveals that they actually hold for matrices over a field with
sufficiently many elements. For example, the proofs in Section 2 hold
for any field with at least $2n$ elements,
and the proofs in Section 3 hold for any field with at least
$(2n-1)^2+1$ elements.  It would be interesting to determine the preserver results for finite fields in general.
It is also curious to consider
rank-preserver questions for matrices over a semiring or a
division ring, in particular, over the skew field $\mathbb H$ of quaternions; e.g., see \cite{Gu}
It can be shown that rank-one Toeplitz matrices in this setting still are either of the form 
$$
A_1=\mu {\bf e}_1{\bf e}_n^\top\quad\mbox{or}\quad
A_2= \mu\sbm{\xi^{n-1}&\ldots & \xi & 1}^\top
\sbm{1&\xi &\ldots &\xi^{n-1}}\lambda
$$
for nonzero scalars $\mu,\lambda$. However, the polynomial approach in the proof of Theorem \ref{main1} requires significant adjustments
(if this is at all possible).

\smallskip

Finally, similar questions can be considered on other structured matrices;
e.g., see \cite{Pan,SJS}.

\bigskip\noindent{\bf \large Acknowledgement:} We thank Dr.\  Vishwa Dewage for some discussion.
Li is an affiliate member
of the Institute for Quantum Computing of the University of Waterloo, 
and also an affiliate member of the Quantum Science \& Engineering Center of George Mason University; his research
was partially supported by the Simons Foundation Grant 851334.

\bigskip\noindent{\bf \large Data availability:} Data sharing not applicable to this article as no datasets
were generated or analyzed during the current study.

\end{document}